\documentclass[12pt,a4paper,final]{amsart}
\usepackage{amssymb}
\usepackage{hyperref}
\usepackage[notcite, notref]{showkeys}
\usepackage[utf8]{inputenc}

\theoremstyle{plain}
\newtheorem{theorem}{Theorem}[section]
\newtheorem{proposition}[theorem]{Proposition}

\newtheorem{lemma}[theorem]{Lemma}

\theoremstyle{definition}

\newtheorem{remark}[theorem]{Remark}
\newtheorem{question}[theorem]{Question}

\newtheorem{notation}[theorem]{Notation}

\theoremstyle{remark}

\numberwithin{equation}{section}
\newcommand{\N}{\mathbb N}
\newcommand{\Z}{\mathbb Z}

\newcommand{\R}{\mathbb R}

\DeclareMathOperator{\GL}{GL}

\DeclareMathOperator{\Ot}{O}
\DeclareMathOperator{\SO}{SO}
\DeclareMathOperator{\SU}{SU}

\DeclareMathOperator{\Spin}{Spin}

\newcommand{\op}{\operatorname}
\newcommand{\Id}{\textup{I}}
\newcommand{\id}{\textup{id}}

\newcommand{\mi}{\mathrm{i}}

\DeclareMathOperator{\diag}{diag}

\DeclareMathOperator{\Spec}{Spec}
\DeclareMathOperator{\spec}{spec}

\DeclareMathOperator{\vol}{vol}

\newcommand{\orders}{\textup{Ord}}

\title[Isospectral spherical space forms and orbifolds]{Isospectral spherical space forms and orbifolds of highest volume}

\author{Alfredo Álzaga}
\address[Alfredo Álzaga]{Departamento de Matemática, Universidad Nacional del Sur (UNS), Bahía Blanca, Argentina}
\email{aalzaga@uns.edu.ar}

\author{Emilio~A.~Lauret}
\address[Emilio~A.~Lauret]{Instituto de Matemática (INMABB), Departamento de Matemática, Universidad Nacional del Sur (UNS)-CONICET, Bahía Blanca, Argentina.}
\email{emilio.lauret@uns.edu.ar}

\thanks{The second author was supported by grants from FONCyT (PICT-2019-01054), SGCYT--UNS (24/L117 y 24/L126), and CONICET (PIP 11220210100343CO)}
\date{\today}

\begin{document}

\begin{abstract}
We prove that $\operatorname{vol}(S^{d})/8$ is the highest volume of a pair of $d$-dimensional isospectral and non-isometric spherical orbifolds for any $d\geq5$. 
Furthermore, we show that
$\operatorname{vol}(S^{2n-1})/11$ is the highest volume of a pair of $(2n-1)$-dimensional isospectral and non-isometric spherical space forms
if either 
$n\geq11$ and $n\equiv 1\pmod 5$, 
or $n\geq7$ and $n\equiv 2\pmod 5$, 
or $n\geq3$ and $n\equiv 3\pmod 5$. 
\end{abstract} 

\maketitle

\tableofcontents
	
\section{Introduction}

Akira Ikeda constructed in 1980 (\cite{Ikeda80_isosp-lens}) early examples of isospectral Riemannian manifolds, namely, families of increasing size of lens spaces (odd-dimensional Riemannian manifolds with constant sectional curvature and cyclic fundamental groups) whose naturally associated Laplace--Beltrami operators have the same spectra (collection of eigenvalues counted with multiplicities). 
Only isospectral flat tori (found firstly by Milnor in the one-page article \cite{Milnor64}) and isospectral Riemann surfaces constructed by Vign\'{e}ras~\cite{Vigneras80} were previously known. 

At that time, there was a big interest in finding similar examples with `minimal properties'.
We next list some instances. 

After the pair of $16$-dimensional flat tori found by Milnor~\cite{Milnor64}, several examples with less dimension have since emerged. 
See \cite[page xxix]{ConwaySloane-book} for a brief tour of the advances, or \cite{RowlettNilssonRydell-flattori} for an enjoyable narration of the developments. 
The final solution was given by Schiemann, finding in \cite{Schiemann90} many pairs in dimension $4$ (with the help of a computer, in 1990!) and proving in \cite{Schiemann97} that there are no examples in dimension $3$. 

Similarly, Ikeda proved that the pair of $5$-dimensional isospectral lens spaces from \cite{Ikeda80_isosp-lens} has minimal dimension among spherical space forms (Riemannian manifolds of positive constant sectional curvature). 
Indeed, on the one hand, it is obvious that there are no even-dimensional or either $1$-dimensional isospectral examples; on the other hand, he proved in \cite{Ikeda80_3-dimI} that there are no isospectral examples in dimension $3$.

As a third instance, the genus of Vign\'{e}ras' examples were enormous. 
Buser~\cite{Buser-book}, using the powerful Sunada's method \cite{Sunada85}, constructed isospectral  Riemann surfaces of genus $5$ and any $g\geq7$. 
Examples with genus $4$ and $6$ were later constructed by Brooks and Tse~\cite{BrooksTse87}. 
Currently, there are no known examples of non-isometric isospectral Riemann surfaces with genus $2$ or $3$.

Interestingly, for any fixed dimension $d\geq2$, the volume of hyperbolic (i.e., constant sectional curvature $-1$) $d$-manifolds or $d$-orbifolds are uniformly bounded by below by a positive real number. 
The problem of determining the hyperbolic manifold or orbifold of lowest volume in a fixed dimension is open; see \cite{BelolipetskyICM} and the references therein for further details. 

Linowitz and Voight~\cite{LinowitzVoight15} studied the problem of determining the lowest area and volume of (strongly) isospectral and non-isometric hyperbolic orbifolds and manifolds of dimensions 2 and 3 within certain nice classes of arithmetic groups. 
Notice there is no analogous problem in the flat case (constant sectional curvature $0$) because scaling a single pair of isospectral compact flat manifolds provides an isospectral pair of any given volume. 

The main goal of this article is to study the analogous problem in the spherical case, that is, among compact Riemannian manifolds with constant sectional curvature $1$. 
Spectral problems for this kind of manifolds have received some attention in the last years (see \cite{LMR-SaoPaulo} and the references therein for details). 

A \emph{spherical space form} is a Riemannian manifold with constant sectional curvature $1$. 
Any of them can be obtained as a quotient $S^d/\Gamma$ of the $d$-dimensional unit sphere $S^d$ by a discrete subgroup $\Gamma$ of the isometry group $\op{Iso}(S^d)\simeq\Ot(d+1)$ acting freely on $S^d$. 
Its fundamental group is isomorphic to $\Gamma$. 
When $\Gamma$ is cyclic and $d$ is odd, $S^d/\Gamma$ is (isometric to) a \emph{lens space} (see Subsection~\ref{subsec:lens} for further details). 

More generally, any \emph{spherical orbifold} is (in this article) of the form $S^d/\Gamma$ with $\Gamma$ a discrete subgroup of $\Ot(d+1)$ that may have fixed points in $S^d$. 
When $\Gamma$ is a cyclic subgroup of $\SO(d+1)$, $S^d/\Gamma$ is (isometric to) a \emph{lens orbifold} (see Subsection~\ref{subsec:lens} for further details). 

The volume of a spherical orbifold $S^d/\Gamma$ is given by 
\begin{equation}\label{eq:vol}
\vol(S^d/\Gamma)=\frac{\vol(S^d)}{|\Gamma|},
\end{equation}
where $|\Gamma|$ stands for the number of elements in $\Gamma$, which is necessarily finite since $\Gamma$ is a discrete subgroup of a compact Lie group. 
Obviously, the volume of a $d$-dimensional spherical orbifold is bounded by above by $\vol(S^d)$. 
Therefore, the spherical analog of the problem of finding the isospectral hyperbolic orbifolds of smallest volume is to find the isospectral spherical orbifolds of largest volume.

\begin{question}\label{question:highestvolume-orbifold}
What is the largest volume of an isospectral pair of $d$-dimensional non-isometric spherical orbifolds?
\end{question}

Linowitz and the second author observed in \cite[Thm.~1.5]{LauretLinowitz} that a pair of almost conjugate subgroups in $\SO(6)$ constructed by Rossetti, Schueth and Weilandt~\cite{RossettiSchuethWeilandt08} provides, for any $d\geq5$, a pair of isospectral and non-isometric spherical orbifolds $S^d/\Gamma$ and $S^d/\Gamma'$   with $\Gamma\simeq\Gamma'\simeq (\Z/2\Z)^3$, thus $|\Gamma|=|\Gamma'|=8$.
This forces for any $d\geq5$ that $\vol(S^d)/8$ is a lower bound for the highest volume requested in Question~\ref{question:highestvolume-orbifold}.  
The first main result in this article confirms that such example realizes the maximum volume when $d\geq 5$. 

\begin{theorem}\label{thm-main:orbifolds}
For any $d\geq5$, the highest volume of an isospectral pair of $d$-dimensional non-isometric spherical orbifolds is $\vol(S^d)/8$. 
\end{theorem}

The proof shows that if $S^d/\Gamma$ and $S^d/\Gamma'$ are isospectral spherical orbifolds with $|\Gamma|=|\Gamma'|\leq 7$, then $\Gamma$ and $\Gamma'$ are conjugate in $\Ot(d+1)$, and consequently $S^d/\Gamma$ and $S^d/\Gamma'$ are isometric.

It is worth mentioning that Theorem~\ref{thm-main:orbifolds} is \emph{currently optimal} in the sense that it is not known so far whether there exist isospectral and non-isometric spherical orbifolds of dimension $d\leq 4$. 
Actually, Bari and Hunsicker~\cite{ShamsHunsicker17} established the non-existence for lens orbifolds. 
Furthermore, V\'{a}squez~\cite[Prop.~2.4]{Vasquez18} proved that there are no strongly isospectral spherical orbifolds, which is equivalent by \cite{Pesce95} to the following fact: 
two almost conjugate subgroups of the double cover $\Spin(4)\simeq \SU(2)\times\SU(2)$ of $\SO(4)$ are necessarily conjugate.

We next move to the manifold case.

\begin{question}\label{question:highestvolume-manifold}
What is the largest volume of an isospectral pair of $d$-dimensional non-isometric spherical space forms?
\end{question}

The condition for a finite subgroup $\Gamma\subset \Ot(d+1)$ acting without fixed points in $S^d$ is a great obstruction. 
Indeed, the classification of spherical space forms, finished by Wolf in \cite{Wolf-book}, is quite technical and difficult to state. 

If $d$ is even, the only $d$-dimensional spherical space forms are $S^d$ and $P^d(\R)$, which are not isospectral. 
Furthermore, it was shown in \cite{Ikeda80_isosp-lens} that there are no $3$-dimensional isospectral and non-isometric spherical space forms.
Therefore, Question~\ref{question:highestvolume-manifold} makes sense only for $d\geq5$ odd. 

It was conjectured in \cite{LauretLinowitz} that isospectral and non-isometric pairs of spherical space forms of highest volume are necessarily lens spaces. 
The conjecture was confirmed (\cite[Thm.~1.7]{LauretLinowitz}) for any $d=2n-1$ with $n\equiv 1\pmod 4$, or $n\equiv 1,2,3\pmod 5$, or $n\equiv 1,2,3,4\pmod 6$, or $n\equiv 2,3,4,5,6\pmod 8$, or $n\equiv 2,3,4,5,6,7\pmod 9$, or $n\equiv 2,3,4,5,6,7,8,9\pmod{11}$. 
For instance, the only possible exceptions for $n\in\N$ with $3\leq n\leq 1000$ are $144$ and $935$.

The strategy was as follows. 
On the one hand, \cite[Lem.~4.4]{LauretLinowitz} implies the following:  
\begin{quote}
\it
if $S^d/\Gamma$ and $S^d/\Gamma'$ are isospectral and non-isometric spherical space forms with $d\geq5$ odd, then $|\Gamma|=|\Gamma'|\geq 24$ if either $\Gamma$ or $\Gamma'$ is non-cyclic.
\end{quote} 
Consequently, to establish the conjecture for a single dimension $d$, it is sufficient to find a pair of $d$-dimensional isospectral and non-isometric lens spaces $S^d/\Gamma$ and $S^d/\Gamma'$ with $|\Gamma|=|\Gamma'|< 24$. 

On the other hand, the following result allowed the authors to find the required isospectral lens spaces in high dimensions from those in small dimensions found with the help of a computer:

\begin{theorem} \cite[Thm.~4.6]{LauretLinowitz} \label{thm:crece-dimension}
If there is a pair of $(2n-1)$-dimensional isospectral and non-isometric lens spaces with fundamental groups of order $q$, then, for any positive integer $r$, there are $\big(2n+r\varphi(q)-1\big)$-dimensional isospectral and non-isometric lens spaces with fundamental groups of order $q$. 
\end{theorem}
Here, $\varphi(q)$ stands for the Euler's totient function, that is, $\varphi(q)=|(\Z/q\Z)^\times|$.

Ikeda~\cite{Ikeda80_3-dimI} found the isospectral and non-isometric pair $L(11;1,2,3)$ and $L(11;1,2,4)$ of $5$-dimensional lens spaces. 
Theorem~\ref{thm:crece-dimension} implies that, for any $r\in\N$, there are pairs of $(5+10r)$-dimensional lens spaces that are isospectral and non-isometric.

With the help of a computer, in \cite{LauretLinowitz}, it was established for $n\leq 14$ that there are isospectral and non-isometric $(2n-1)$-lens spaces with fundamental group of order $11$ if and only if $n\in\{3,7,8,11,12,13\}$. 

It follows that the highest volume of a pair of isospectral and non-isometric $(2n-1)$-dimensional spherical space forms is at most $\vol(S^{2n-1})/11$ if $n\equiv k\pmod{5}$ for some $k\in\{3,7,8,11,12,13\}$. 
We prove in Theorem~\ref{thm:q<10} that two isospectral $(2n-1)$-dimensional lens spaces with volume $\geq \vol(S^{2n-1})/10$ are necessarily isometric. 
Both facts combined answer Question~\ref{question:highestvolume-manifold} for infinitely many odd dimensions, essentially, $60\%$ of the cases. 
This is precisely the second main theorem of the article. 

\begin{theorem} \label{thm-main:manifolds}
If $n\in\{3,7,8\}$ or $n=5r+k$ for some integer $r\geq2$ and $k\in\{1,2,3\}$, then $\vol(S^{2n-1})/11$ is the highest volume of an isospectral pair of $(2n-1)$-dimensional non-isometric spherical space forms. 
\end{theorem}

The paper is organized as follows. 
Section~\ref{sec:preliminaries} contains preliminaries on the spectrum of spherical orbifolds. 
In Section~\ref{sec:manifolds} we prove Theorem~\ref{thm:q<10}, which implies Theorem~\ref{thm-main:manifolds}.
The first main theorem, Theorem~\ref{thm-main:orbifolds}, is established in Section~\ref{sec:orbifolds}. 
Section~\ref{sec:errata} is devoted to correcting a wrong statement in \cite{LauretLinowitz}.

\subsection*{Acknowledgments}
The authors thank Juliana Cornago for her assistance in proofreading this article, and the referee for a 
thorough reading and several detailed corrections.

\section{Preliminaries}\label{sec:preliminaries}

The goal of this section is to introduce the main tool: the spectral generating function associated to a spherical orbifold.  

\subsection{Spectra of spherical orbifolds}
It is well known that the spectrum of the Laplace--Beltrami operator associated to the $d$-dimensional unit sphere $S^d$ is given as follows: the $k$-th distinct eigenvalue is $\lambda_k:=k(k+d-1)$ with multiplicity $\dim \mathcal H_k$, where $\mathcal H_k$ is the space of complex harmonic homogeneous polynomials on $(d+1)$-variables of degree $k$.
That is, 
\begin{equation*}
\Spec(S^d)
= \Big\{\!\!\Big\{ 			\underbrace{\lambda_k,\dots,\lambda_k}_{ \dim \mathcal H_k\text{-times}} : k\geq0 
\Big\}\!\!\Big\}
.
\end{equation*}
Although it will not be used, we recall that
$
\dim\mathcal H_k=\binom{k+d}{d}-\binom{k+d-2}{d}.
$

The group $\Ot(d+1)$ acts on $\mathcal H_k$ as follows: $(a\cdot P)(z)=P(a^{-1}z)$ for $a\in\Ot(d+1)$ and $P\in\mathcal P_k$, where $z$ denotes the column vector with the variables $z_1,\dots,z_{d+1}$ and $a^{-1}z$ stands for the matrix multiplication. 
Indeed, $\mathcal H_k$ is an irreducible representation of $\Ot(d+1)$. 

In this article, we call a \emph{spherical orbifold} a quotient of the form $S^d/\Gamma$ with $\Gamma$ a finite subgroup of $\Ot(d+1)$, endowed with the Riemannian metric induced by the round metric on $S^d$ with constant sectional curvature $+1$ (some authors include in the definition of spherical orbifolds the quotients of unit balls by orthogonal actions of finite groups). 
See \cite{Gordon12-orbifold} for details. 
If two finite subgroups $\Gamma$ and $\Gamma'$ of $\Ot(d+1)$ are conjugate, then the corresponding orbifolds $S^d/\Gamma$ and $S^d/\Gamma'$ are isometric.

Let $\Gamma$ be a finite subgroup of $\Ot(d+1)$. 
The spectrum of the spherical orbifold $S^{d}/\Gamma$ is given by 
\begin{equation}\label{eq:Spec(S^d/Gamma)}
\Spec(S^d/\Gamma)
= \Big\{\!\!\Big\{ 			\underbrace{\lambda_k,\dots,\lambda_k}_{ \dim \mathcal H_k^\Gamma\text{-times}} : k\geq0 
\Big\}\!\!\Big\}
.
\end{equation}
Here, $\mathcal H_k^\Gamma$ denotes the $\Gamma$-invariant elements in $\mathcal H_k$ with respect to the action defined above. 

As an immediate consequence of \eqref{eq:Spec(S^d/Gamma)}, one has that two $d$-dimensional spherical space forms $S^{d}/\Gamma$ and $S^{d}/\Gamma'$ are isospectral (i.e., $\Spec(S^d/\Gamma)=\Spec(S^d/\Gamma')$) if and only if 
\begin{equation}\label{eq:dimH_k^Gamma}
\dim \mathcal H_k^{\Gamma} = \dim \mathcal H_k^{\Gamma'}
\qquad\text{for all $k\geq0$}. 
\end{equation}

\subsection{Spectral invariants}

In \cite{IkedaYamamoto79}, Ikeda and Yamamoto associated to $S^d/\Gamma$ the generating function
\begin{equation}
F_{\Gamma}(z)=\sum_{k\geq0} \dim\mathcal H_k^\Gamma\, z^k. 
\end{equation}
The great importance of this function is due to the next fact, which follows from
\eqref{eq:dimH_k^Gamma}:
\begin{equation}
\Spec(S^d/\Gamma)=\Spec(S^d/\Gamma')
\qquad\Longleftrightarrow\qquad F_{\Gamma}(z)=F_{\Gamma'}(z)
. 
\end{equation}
Consequently, any property determined by $F_{\Gamma}(z)$ will be a spectral invariant, for instance, the set of poles (we will see next that it is a rational function), their corresponding multiplicities, principal parts, etc.

Using Molien's formula, one obtains (see \cite[(2.4)]{Wolf01})
\begin{equation}\label{eq:Molien}
F_{\Gamma}(z)=\frac{1-z^2}{|\Gamma|} \sum_{\gamma\in\Gamma} \frac{1}{\det(\Id_{d+1}-z\gamma)},
\end{equation}
where $\det(\Id_{d+1}-z\gamma)$ stands for the determinant of $\Id_{d+1}-z\gamma\in\Ot(d+1)$.
This expression was first found by Ikeda~\cite[Thm.~2.2]{Ikeda80_3-dimI}.

\begin{remark}\label{rem:volume}
It is well known that two isospectral compact Riemannian orbifolds share the same dimension and volume. 
Consequently, if $\Spec(S^d/\Gamma)=\Spec(S^{d'}/\Gamma')$, then $d=d'$ and $|\Gamma|=|\Gamma'|$ by \eqref{eq:vol}.

Curiously, both facts can be elegantly shown from \eqref{eq:Molien} for spherical orbifolds. 
In other words, $d$ and $|\Gamma|$ are determined by $F_{\Gamma}(z)$, as we next observe. 

It is clear that $F_{\Gamma}(z)$ has a pole at $z=1$ of order $d$ since
$$
F_{\Gamma}(z)=\frac{1}{|\Gamma|} \frac{1-z^2}{(1-z)^{d+1}} + \frac{1-z^2}{|\Gamma|} \sum_{\gamma\in\Gamma,\,\gamma\neq\Id_{d+1}} \frac{1}{\det(\Id_{d+1}-z\gamma)}
,
$$
and $\det(\Id_{d+1}-z\gamma)=(1-z)^{d+1}$ if and only if $\gamma=\Id_{d+1}$. 
Consequently, $d$ is determined by $F_{\Gamma}(z)$. 
Moreover, the $(-d)$-th coefficient of its Laurent series at $z=1$ is given by 
\begin{align*}
\lim_{z\to1} (z-1)^{d}F_{\Gamma}(z) &
=\lim_{z\to1} (-1)^{d}\frac{1+z}{|\Gamma|} 
= \frac{2(-1)^{d}}{|\Gamma|} ,
\end{align*}  
thus $|\Gamma|$ is determined by $F_{\Gamma}(z)$. 
\end{remark}

A spherical space form is a spherical orbifold $S^{d}/\Gamma$ with $\Gamma$ acting freely on $S^{d}$, or equivalently, $1$ is not an eigenvalue of any non-trivial element in $\Gamma$. 
It is well known that if $d=2n$, the only $d$-dimensional spherical space forms are $S^d$ and $P^d(S)=S^{d}/\{\pm\Id_{d+1}\}$.

\begin{remark}
Let $S^d/\Gamma$ be an odd-dimensional spherical space form.
We claim that $\det(\Id_{d+1}-z\gamma)=\det(z\Id_{d+1}-\gamma)=:\det(z-\gamma)$ for all $\gamma\in\Gamma$, hence \eqref{eq:Molien} becomes
\begin{equation}\label{eq:Molien2}
F_{\Gamma}(z)=\frac{1-z^2}{|\Gamma|} \sum_{\gamma\in\Gamma} \frac{1}{\det(z-\gamma)}.
\end{equation}

Let $\gamma$ be an element in $\Gamma$ of order $k$. 
We write $\spec(\gamma)$ for the multiset of eigenvalues of the orthogonal matrix $\gamma$. 
We have that $\det(\Id_{d+1}-z\gamma)=\prod_{\lambda\in\spec(\gamma)}(1-z\lambda)$. 
Since $1\notin \spec(\gamma')$ for all $\gamma'\neq\Id_{d+1}$ in $\Gamma$, 
it follows that the elements in $\spec(\gamma)$ are primitive $k$-th roots of unity. 
Moreover, since $\gamma$ is an orthogonal matrix, its eigenvalues come in conjugate pairs in the sense that if $\lambda\neq\pm1$ is in $\spec(\gamma)$, then $\lambda$ and $\bar \lambda$ have the same multiplicity in $\spec(\gamma)$. 
We conclude when $k\geq3$ that 
\begin{equation}
\begin{aligned}
\det(\Id_{d+1}-z\gamma)&
=\prod_{\lambda\in\spec(\gamma): \op{Im}(\lambda)>0}(1-z\lambda)(1-z\bar\lambda)
=\prod_{\lambda\in\spec(\gamma): \op{Im}(\lambda)>0}(\bar \lambda-z)(\lambda-z)
\\ &
=\prod_{\lambda\in\spec(\gamma): \op{Im}(\lambda)>0}(z-\lambda)(z-\bar\lambda)
=\det(z-\gamma)
.
\end{aligned}
\end{equation}
If $k=2$, then $\gamma=-\Id_{d+1}$, thus $\det(\Id_{d+1}-z\gamma)=(1+z)^{d+1} =\det(z-\gamma)$.
If $k=1$, then $\gamma=\Id_{d+1}$, thus $\det(\Id_{d+1}-z\gamma)=(1-z)^{d+1}=(z-1)^{d+1}= \det(z-\gamma)$ since $d+1$ is even. 

\end{remark}

Ikeda found the following spectral invariants for odd-dimensional spherical space forms (Prop.~2.7 and Cor.~2.8 in \cite{Ikeda80_3-dimI}, Prop.~1.5(3) in \cite{Ikeda80_3-dimII}). 

\begin{proposition}\label{prop:specinvariantsIkeda}
Let $S^{2n-1}/\Gamma$ be a spherical space form. 
The following properties are determined by $F_{\Gamma}(z)$: 
\begin{enumerate}
\item\label{item:orders} 
The set of orders of the elements in $\Gamma$: $\orders(\Gamma):=\{k\in\N: \exists  \gamma\in\Gamma\text{ of order $k$}\}$. 

\item\label{item:maxmultip}
For $k\in \orders(\Gamma)$,
$$
\max_{\gamma\in\Gamma(k)} \{ \text{multiplicity of $\xi_{k}$ as an eigenvalue of $\gamma$} \}
,
$$ 
where $\xi_k=e^{2\pi\mi/k}$, and $\Gamma(k)$ stands for the set of elements in $\Gamma$ with order $k$. 

\item\label{item:F^k}
For $k\in \orders(\Gamma)$, the generating function
$$
F_{\Gamma}^{(k)}(z)
:=\sum_{\gamma\in\Gamma(k)} \frac{1}{\det(\Id_{2n}-z\gamma)}
=\sum_{\gamma\in\Gamma(k)} \frac{1}{\det(z-\gamma)}
. 
$$
\end{enumerate}
\end{proposition}

Consequently, two odd-dimensional isospectral spherical space forms $S^{2n-1}/\Gamma$ and $S^{2n-1}/\Gamma'$ satisfy $\orders(\Gamma)=\orders(\Gamma')$ and, for any $k\in\orders(\Gamma)$, $F_{\Gamma}^{(k)}(\Gamma)= F_{\Gamma}^{(k)}(\Gamma')$ and 
the highest multiplicity of $\xi_k$ as an eigenvalue of  the elements in $\Gamma(k)$ and $\Gamma'(k)$ coincide.

\subsection{Lens orbifolds}\label{subsec:lens}

A \emph{lens space} is a quotient of the form $L(q;s)=S^{2n-1}/\Gamma_{q,s}$ for some $q\in\N$ and $s=(s_1,\dots,s_n)\in\Z^n$ satisfying $\gcd(q,s_j)=1$ for all $1\leq j\leq n$, where 
\begin{equation}\label{eq:Gamma_qs}
\Gamma_{q,s}
= \Big\langle
\begin{pmatrix}
R(\tfrac{2\pi s_1}{q}) \\
&\ddots\\
&&R(\tfrac{2\pi s_n}{q}) 
\end{pmatrix}
\Big\rangle
\quad\text{ and }\quad
R(\theta)=  \begin{pmatrix}
\cos(\theta)&\sin(\theta)\\ -\sin(\theta)&\cos(\theta)\end{pmatrix}
.
\end{equation}
The conditions $\gcd(q,s_j)=1$ for all $j$ ensure that $|\Gamma_{q,s}|=q$ and $\Gamma_{q,s}$ acts freely on $S^{2n-1}$, thus $L(q;s)$ is a $(2n-1)$-dimensional differentiable manifold with fundamental group isomorphic to $\Gamma_{q,s}\simeq\Z/q\Z$.
We will always endow a lens space $L(q;s)$ with the round Riemannian metric (with constant sectional curvature one), which makes it a spherical space form.

We next relax the above conditions to introduce (what we call in this article) lens orbifolds. 

We fix a dimension $d\in\N$. 
We write $n=\lfloor\frac{d+1}{2}\rfloor$, so $d=2n$ or $d=2n-1$, depending on whether $d$ is even or odd, respectively. 
We consider the maximal torus $T$ of $\SO(d+1)$ given by 
\begin{equation*}
T=
\big\{
t(\theta_1,\dots, \theta_{n})
:\theta_1,\dots, \theta_{n} \in\R
\big\}  ,
\end{equation*}
where 
\begin{equation*}
t(\theta_1,\dots, \theta_{n})=
\begin{cases}
\diag\big(
R(\theta_1),\dots,
R(\theta_{n}), 1\big)
	&\text{if $d$ is even},
\\
\diag \big(
R(\theta_1),\dots,
R(\theta_{n})\big)
	&\text{if $d$ is odd}.
\end{cases}
\end{equation*}

A \emph{lens orbifold} is a quotient of the form $L_d(q;s):=S^d/\Gamma_{q,s}$, with $\Gamma_{q,s}=\langle t(\frac{2\pi s_1}{q}, \dots, \frac{2\pi s_n}{q})\rangle$  for some $q\in\N$ and $s\in\Z^n$ satisfying $\gcd(q,s_1,\dots,s_n)=1$.  
The last condition ensures that $|\Gamma_{q,s}|=q$. 
We endow the orbifold $L_d(q;s)$ with the round Riemannian metric induced by the constant sectional curvature one metric on $S^d$. 
If the dimension $d$ is odd and is clear from the context, we abbreviate $L_d(q;s)$ by $L(q;s)$, coinciding with the notation of lens spaces.

The multiset of eigenvalues of $t(\frac{2\pi s_1}{q}, \dots, \frac{2\pi s_n}{q})^k=t(\frac{2k\pi s_1}{q}, \dots, \frac{2k\pi s_n}{q})$ is
\begin{equation}\label{eq:eigenvalues}
\begin{cases}
\{\!\{\xi_q^{k s_j}, \xi_q^{-k s_j}:1\leq j\leq n\}\!\}
	&\quad\text{if $d$ is odd},\\
\{\!\{\xi_q^{k s_j}, \xi_q^{-k s_j}:1\leq j\leq n\}\!\}\cup\{\!\{1\}\!\}
	&\quad\text{if $d$ is even},\\
\end{cases}
\end{equation} 
where $\xi_q=e^{2\pi\mi /q}$, a primitive $q$-th root of unity. 
Now, \eqref{eq:Molien} gives
\begin{equation}\label{eq:MolienLens}
\begin{aligned}
F_{\Gamma_{q,s}}(z)&
=\frac{1-z^2}{q} \sum_{k=0}^{q-1} 
	\left(\prod_{j=1}^n \frac{1}{(z-\xi_q^{ks_j})(z-\xi_q^{-ks_j})}
	\right) 
	\frac{1}{(1-z)^{d-2n+1}}
	.
\end{aligned}
\end{equation}

Let $\Gamma$ be any cyclic subgroup of $\SO(d+1)$. 
The spherical orbifold $S^d/\Gamma$ is isometric to some lens orbifold $L_d(q;s)$. 
Indeed, by setting $q=|\Gamma|$ and letting $\gamma$ be any generator of $\Gamma$, it is well known that $\gamma$ is conjugate to some element (of order $q$) in the maximal torus $T$, say $t(\frac{2\pi s_1}{q}, \dots, \frac{2\pi s_n}{q})$ for some $s_1,\dots,s_n\in\Z$ necessarily satisfying $\gcd(q,s_1,\dots,s_n)=1$ (otherwise, it has order strictly less than $q$). 
Thus, we obtain that $S^d/\Gamma$ is isometric to $L(q;s)$.  
Consequently, the next result classifies the isometry classes of lens orbifolds (see e.g.\ \cite[Prop.~2.9]{Lauret-spec0cyclic}).

\begin{proposition}\label{prop:lensisometries}
Given two $d$-dimensional lens orbifolds $L_d(q;s)$ and $L_d(q';s')$, the following conditions are equivalent:
\begin{itemize}
\item $L_d(q;s)$ and $L_d(q';s')$ are homeomorphic. 

\item $L_d(q;s)$ and $L_d(q';s')$ are isometric. 

\item $q=q'$, and there are $l\in\Z$ prime to $q$, a permutation $\sigma$ of $\{1,\dots,n\}$, and $\epsilon_1,\dots,\epsilon_n\in\{\pm1\}$ such that 
\begin{equation}\label{eq:lensisometry}
s_{\sigma(j)}'\equiv l\epsilon_j s_j \pmod{q}\qquad\forall \,j=1,\dots,n .
\end{equation}
\end{itemize}
\end{proposition}

We next obtain a useful consequence of Proposition~\ref{prop:lensisometries} after introducing some notation for lens orbifolds. 
\begin{notation}\label{notation}
We fix a dimension $d\in\N$. 
For $q\in\N$, $s_1,\dots,s_l\in\Z$ and $n_1,\dots,n_l\in\N$ satisfying $\gcd(q,s_1,\dots,s_l)=1$ and $n=\lfloor\frac{d+1}{2}\rfloor=n_1+\dots+n_l$, we abbreviate 
$$
L_d(q;s_1^{n_1},\dots ,s_l^{n_l})
= L_d\big(q; 
\underbrace{s_1,\dots,s_1}_{\text{$n_1$-times}}
,\dots,
\underbrace{s_l,\dots,s_l}_{\text{$n_l$-times}}
\big)
,
$$ 
which is a $d$-dimensional lens orbifold with volume $\vol(S^d)/q$. 

Similarly, for $q\in\N$, $s_1,\dots,s_l\in\Z$ and $n_1,\dots,n_l\in\N_0:=\N\cup \{0\}$ satisfying $\gcd(s_i,q)=1$ for all $i$, we abbreviate 
$$
L(q;s_1^{n_1},\dots ,s_l^{n_l})
= L\big(q; 
\underbrace{s_1,\dots,s_1}_{\text{$n_1$-times}}
,\dots,
\underbrace{s_l,\dots,s_l}_{\text{$n_l$-times}}
\big)
,
$$ 
which is a $(2n-1)$-dimensional lens space with $n=n_1+\dots+n_l$ with fundamental group of order $q$. 
\end{notation}

\begin{lemma}\label{lem:exponentemaximo}
Let $q$ be a positive integer satisfying $q\geq3$, and write $q_0=\varphi(q)/2$.   
Let $s_1<s_2<\dots<s_{q_0}$ be all positive integers less than $q/2$ and prime to $q$ (so $s_1=1$). 
For any $(2n-1)$-dimensional lens space with fundamental group of order $q$, there are $n_1,\dots,n_{q_0}\in\N_0$  with $n=n_1+\dots+n_{q_0}$ such that 
$L$ is isometric to $L\big(q;s_1^{n_1},\dots,s_{q_0}^{n_{q_0}}\big)$
and $n_1\geq \max(n_2,\dots,n_{q_0})$. 
\end{lemma}

\begin{proof}
The third condition in Proposition~\ref{prop:lensisometries} allows us to choose the parameters of $L$ among $\{s_1,\dots,s_{q_0}\}$. 
In other words, there are $n_1,\dots,n_{q_0}\in\N$ such that $L\simeq L\big(q;s_1^{n_1},\dots,s_{q_0}^{n_{q_0}}\big)$. 

Suppose $n_i=\max(n_1,\dots,n_{q_0})$. 
Let $l\in\Z$ be an inverse of $s_i$ module $q$, that is, $ls_i\equiv 1\pmod q$. 
Proposition~\ref{prop:lensisometries} implies that
$$
L
\simeq L\big(q;(ls_1)^{n_1},\dots,(ls_{q_0})^{n_{q_0}}\big). 
$$
Clearly, $\{\pm ls_1,\dots,\pm ls_{q_0}\}$ is a representative set of $(\Z/q\Z)^\times$.
Hence, there is a permutation $\sigma$ of $\{1,\dots,q_0\}$ such that $ls_j\equiv \epsilon_j s_{\sigma(j)}\pmod q$ for some $\epsilon_1,\dots,\epsilon_{q_0} \in\{\pm1\}$. 
Write $\tau=\sigma^{-1}$.
Note that $\tau(1)=i$. 
Now, Proposition~\ref{prop:lensisometries} yields that
$$
L
\simeq L\big(q;s_1^{n_i},s_2^{n_{\tau(2)}}, \dots,s_{q_0}^{n_{\tau(q_0)}}\big),
$$
which completes the proof. 
\end{proof}

\begin{remark}\label{rem:highestmultiplicity}
Let $L(q;s_1^{n_1},\dots,s_{q_0}^{n_{q_0}})=S^{2n-1}/\Gamma_{q,s}$ be a spherical space form written as in Lemma~\ref{lem:exponentemaximo}. 
Looking at the eigenvalues of the elements in $\Gamma_{q,s}$ in \eqref{eq:eigenvalues}, it follows that 
$$
n_1=\max_{\gamma\in\Gamma(q)} \{ \text{multiplicity of $\xi_{q}$ as an eigenvalue of $\gamma$} \}. 
$$ 
Hence, $n_1$ is determined by $F_{\Gamma_{q,s}}(z)$ by Proposition~\ref{prop:specinvariantsIkeda}\eqref{item:maxmultip}. 
\end{remark}

\section{Highest volume of spherical space forms}\label{sec:manifolds}

The main goal of this section is to prove that two isospectral lens spaces with fundamental group of order $q\leq 10$ are necessarily isometric (Theorem~\ref{thm:q<10}). 
However, we prove it for all positive integers $q$ such that $\varphi(q)\leq 6$, where $\varphi$ stands for the Euler's totient function (i.e., $\varphi(q)=|(\Z/q\Z)^\times|$).

\begin{proposition}\label{prop:varphi(q)<=6}
Let $q$ be a positive integer such that $\varphi(q)\leq 6$.
Then, two isospectral lens spaces with fundamental groups of order $q$ are necessarily isometric. 
\end{proposition}

\begin{proof}
Let $L=S^{2n-1}/\Gamma$ be a $(2n-1)$-dimensional lens space with $|\Gamma|=q$. 
The goal is to show that $F_{\Gamma}(z)$ determines $L$ up to isometry, or equivalently, $\Gamma$ up to conjugation in $\Ot(2n)$. 
We recall from Remark~\ref{rem:volume} that $n$ and $q$ are determined by $F_{\Gamma}(z)$. 

If $\varphi(q)\leq 2$, then there is up to isometry exactly one $(2n-1)$-dimensional lens space with fundamental group of order $q$.  
It remains to consider the cases $\varphi(q)=4$ and $\varphi(q)=6$ since $\varphi(q)$ is even for any $q\geq3$. 

Assume $\varphi(q)=4$. 
Let $s_2$ be the only integer prime to $q$ satisfying $1<s_2<q/2$. 
Lemma~\ref{lem:exponentemaximo} ensures that any $(2n-1)$-dimensional lens space with fundamental group of order $q$ is isometric to 
$$
L(q;1^{n-h},s_2^{h})
\quad 
\text{for some }h\in\{0,1,\dots,\lfloor n/2\rfloor\}.
$$
Since $n-h$ is determined by $F_{\Gamma}(z)$ (see Remark~\ref{rem:highestmultiplicity}), so is $h$ and consequently the isometry class of $L$. 

From now on we suppose that $\varphi(q)=6$, that is, $q\in\{7,9,14,18\}$. 
We denote by $1=s_1<s_2<s_{3}$ all positive integers less than $q/2$ and prime to $q$. 
One can easily check that 
\begin{align}\label{eq-varphi=6:uv}
s_2^2&\equiv \pm s_3\pmod q,&
s_3^2&\equiv \pm s_2\pmod q&
s_2s_3&\equiv \pm 1\pmod q.
\end{align}

By Lemma~\ref{lem:exponentemaximo}, we do not lose generality by assuming that 
$
L=L\big(q;1^{n_1}, s_2^{n_2},s_{3}^{n_{3}}\big)
$
for some non-negative integers $n_1,n_2,n_{3}$ satisfying $n=n_1+n_2+n_{3}$ and $n_1\geq\max(n_2,n_{3})$.
Moreover, we can also assume that
\begin{align}\label{eq-varphi=6:conditions}
\text{either}\qquad
n_1>\max(n_2,n_3)
\qquad\text{or}\qquad 
n_1=n_2> n_3.
\qquad\text{or}\qquad 
n_1=n_2=n_3.
\end{align}
Notice the case $n_1=n_3>n_2$ does not appear above since Proposition~\ref{prop:lensisometries} and \eqref{eq-varphi=6:uv} imply that $L(q;1^{n_1},s_2^{n_2},s_3^{n_1})\simeq L(q;1^{n_1}, s_2^{n_1},s_3^{n_2})$ after multiplying the parameters of $L(q;1^{n_1},s_2^{n_2},s_3^{n_1})$ by $s_2$.

If $\gamma$ is any generator of $\Gamma$, then the subset of elements of order $q$ is equal to $\Gamma(q)=\{\gamma^k: \gcd(k,q)=1\}$. 
Hence, the function introduced in Proposition~\ref{prop:specinvariantsIkeda}\eqref{item:F^k} is given by
\begin{equation}\label{eq:Fprincipalsingularpart}
\begin{aligned}
F_{\Gamma}^{(q)}(z) &
= 
\sum_{\substack{0\leq k<q \\ \gcd(k,q)=1}}
\prod_{j=1}^{3} 
	\frac{1}{(z-\xi_q^{ks_j})^{n_j} (z-\xi_q^{-ks_j})^{n_j}} 
= 
2\sum_{i=1}^{3}
\prod_{j=1}^{3} 
	\frac{1}{(z-\xi_q^{s_is_j})^{n_j} (z-\xi_q^{-s_is_j})^{n_j}} 
.
\end{aligned}
\end{equation}
Now, \eqref{eq:Fprincipalsingularpart} and \eqref{eq-varphi=6:uv} imply that 
\begin{equation}\label{eq-varphi=6:Ftilde}
\begin{aligned}
F_{\Gamma}^{(q)}(z) &
= \sum_{i=1}^{3}
\frac{2}{
	(z-\xi_q^{s_i})^{n_1} (z-\xi_q^{-s_i})^{n_1}
	(z-\xi_q^{s_is_2})^{n_2}(z-\xi_q^{-s_is_2})^{n_2}
	(z-\xi_q^{s_is_3})^{n_3}(z-\xi_q^{-s_is_3})^{n_3}
}
\\ & 
=
\frac{2}{
	(z-\xi_q)^{n_1} (z-\xi_q^{-1})^{n_1}
	(z-\xi_q^{s_2})^{n_2}(z-\xi_q^{-s_2})^{n_2}
	(z-\xi_q^{s_3})^{n_3}(z-\xi_q^{-s_3})^{n_3}
}
\\ & \quad +
\frac{2}{
	(z-\xi_q^{s_2})^{n_1} (z-\xi_q^{-s_2})^{n_1}
	(z-\xi_q^{s_3})^{n_2}(z-\xi_q^{-s_3})^{n_2}
	(z-\xi_q)^{n_3}(z-\xi_q^{-1})^{n_3}
}
\\ & \quad +
\frac{2}{
	(z-\xi_q^{s_3})^{n_1} (z-\xi_q^{-s_3})^{n_1}
	(z-\xi_q)^{n_2}(z-\xi_q^{-1})^{n_2}
	(z-\xi_q^{s_2})^{n_3}(z-\xi_q^{-s_2})^{n_3}
}
\\ & 
=
\frac{2}{\Phi_q(z)^{n_1}} P_{\Gamma}(z)
,
\end{aligned}
\end{equation}
where $\Phi_q(z)=(z-\xi_q)(z-\xi_q^{-1}) (z-\xi_q^{s_2})(z-\xi_q^{-s_2}) (z-\xi_q^{s_3})(z-\xi_q^{-s_3})$ is the $q$-th cyclotomic polynomial and 
\begin{equation}
\begin{aligned}
P_{\Gamma}(z)&
= 	(z-\xi_q^{s_2})^{n_1-n_2}(z-\xi_q^{-s_2})^{n_1-n_2}
	(z-\xi_q^{s_3})^{n_1-n_3}(z-\xi_q^{-s_3})^{n_1-n_3}
	\\& +
	(z-\xi_q^{s_3})^{n_1-n_2}(z-\xi_q^{-s_3})^{n_1-n_2}
	(z-\xi_q)^{n_1-n_3}(z-\xi_q^{-1})^{n_1-n_3}
	\\ &+ 
	(z-\xi_q)^{n_1-n_2}(z-\xi_q^{-1})^{n_1-n_2}
	(z-\xi_q^{s_2})^{n_1-n_3}(z-\xi_q^{-s_2})^{n_1-n_3}	
.
\end{aligned}
\end{equation}
Since $F_{\Gamma}(z)$ determines $n_1$ (see Remark~\ref{rem:highestmultiplicity}) and $F_{\Gamma}^{(q)}(z)$ by Proposition~\ref{prop:specinvariantsIkeda}\eqref{item:F^k}, it follows that $P_{\Gamma}(z)$ is also determined by $F_{\Gamma}(z)$.

We next see that $P_{\Gamma}(z)$ determines which condition in \eqref{eq-varphi=6:conditions} occurs. 
Clearly, $P_{\Gamma}\equiv 3$ if and only if $n_1=n_2=n_3$. 
Suppose that $\Gamma$ has parameters $(n_1,n_2,n_3)$ with $n_1>\max(n_2,n_3)$ and $\Gamma'$ has parameters  $(n_1,n_1,n_3')$ with $n_1>n_3'$. 
Since $n=n_1+n_2+n_3=2n_1+n_3'$, we have $n_2+n_3=n_1+n_3'$. 
By writing $r_j=(1-\xi_q^{s_j})(1-\xi_q^{-s_j}) =2-2\cos(\tfrac{2s_j\pi}{q})\in\R$ for any $j=1,\dots,3$, it follows that
\begin{equation}\label{eq-varphi=6:P(1)-P'(1)}
\begin{aligned}
P_{\Gamma}(1)-P_{\Gamma'}(1)&
	= r_1^{n_1-n_2}r_2^{n_1-n_3} 
	+ r_2^{n_1-n_2}r_3^{n_1-n_3} 
	+ r_3^{n_1-n_2}r_1^{n_1-n_3} 
	-\big(r_1^{n_1-n_3'} +r_2^{n_1-n_3'} +r_3^{n_1-n_3'} \big)
\\ &
	= r_1^{n_1-n_2} (r_2^{n_1-n_3} - r_1^{n_1-n_3})
	+ r_2^{n_1-n_2} (r_3^{n_1-n_3} - r_2^{n_1-n_3})
	- r_3^{n_1-n_2} (r_3^{n_1-n_3} - r_1^{n_1-n_3}) 
\\ &
	< r_2^{n_1-n_2} (r_3^{n_1-n_3} - r_1^{n_1-n_3})
	- r_3^{n_1-n_2} (r_3^{n_1-n_3} - r_1^{n_1-n_3}) 
\\ &
	= (r_2^{n_1-n_2} - r_3^{n_1-n_2})
	(r_3^{n_1-n_3} - r_1^{n_1-n_3})<0
\end{aligned}
\end{equation}
since $n_1-n_3'= (n_1-n_2)+(n_1-n_3)$ and $0<r_1<r_2<r_3$. 
We conclude that $F_{\Gamma}(z)$ distinguishes each of the three conditions in \eqref{eq-varphi=6:conditions}. 

We recall that $n_1$ is determined by $F_{\Gamma}(z)$ by Remark~\ref{rem:highestmultiplicity}, as well as $n=n_1+n_2+n_3$ by Remark~\ref{rem:volume}. 
It only remains to show that the parameters $n_2$ and $n_3$ are determined by $F_{\Gamma}(z)$ in each case in \eqref{eq-varphi=6:conditions} separately. 
If $n_1=n_2=n_3$, the assertion is trivial. 
If $n_1=n_2>n_3$, then $n_3=n-2n_1$ is determined by $F_{\Gamma}(z)$. 

We next deal with the case $n_1>\max(n_2,n_3)$. 
Since $n_3=n-n_1-n_2$, we have that 
\begin{equation}
\begin{aligned}
|P_{\Gamma}(\xi_q)|&= 
\left|
(\xi_q-\xi_q^{s_2})^{n_1-n_2} (\xi_q-\xi_q^{-s_2})^{n_1-n_2} (\xi_q-\xi_q^{s_3})^{2n_1+n_2-n} (\xi_q-\xi_q^{-s_3})^{2n_1-n+n_2}
\right|
\\ & 
=\left|
(\xi_q-\xi_q^{s_2})(\xi_q-\xi_q^{-s_2})
\right|^{n_1}
\left|
(\xi_q-\xi_q^{s_3})(\xi_q-\xi_q^{-s_3})
\right|^{2n_1-n}
\left|
	\frac{(\xi_q-\xi_q^{s_3}) (\xi_q-\xi_q^{-s_3})}
	{(\xi_q-\xi_q^{s_2})(\xi_q-\xi_q^{-s_2})} 
\right|^{n_2}
.
\end{aligned}
\end{equation}
Since $n$ and $n_1$ are known and $|(\xi_q-\xi_q^{s_3})(\xi_q-\xi_q^{-s_3})|>|(\xi_q-\xi_q^{s_2})(\xi_q-\xi_q^{-s_2})|$ because $1<s_2<s_3<q/2$, we deduce that $n_2$ is determined by $|P_{\Gamma}(\xi_q)|$.
In fact, the function $n_2\mapsto |P_{\Gamma}(\xi_q)|$ is strictly increasing, and consequently it is injective.
We conclude that $n_2$ and $n_3=n-n_1-n_2$ are determined by $F_{\Gamma}(z)$, and this completes the proof. 
\end{proof}

\begin{theorem}\label{thm:q<10}
Two isospectral lens spaces with fundamental group of order $q\leq 10$ are necessarily isometric. 
\end{theorem}

\begin{proof}
The proof follows immediately from Proposition~\ref{prop:varphi(q)<=6} since $\varphi(q)\leq 6$ for all $q\leq 10$. 
\end{proof}

\section{Highest volume of spherical orbifolds}
\label{sec:orbifolds}
In this section we prove Theorem~\ref{thm-main:orbifolds}, that is, $\vol(S^d)/8$ is the maximum volume of an isospectral and non-isometric pair of $d$-dimensional spherical orbifolds for any $d\geq5$. 
We already know that $\vol(S^d)/8$ is a lower bound due to the pair from \cite{LauretLinowitz} of isospectral and non-isometric spherical space forms $S^d/\Gamma$ and $S^d/\Gamma'$ with $\Gamma\simeq\Gamma'\simeq (\Z/2\Z)^3$. 
The equality follows immediately from the next result. 

\begin{theorem}\label{thm:orbifolds<8}
Fix $d\in\N$, and let $\Gamma,\Gamma'$ be finite subgroups of $\Ot(d+1)$. 
If $S^d/\Gamma$ and $S^d/\Gamma'$ are isospectral and  $|\Gamma|=|\Gamma'|\leq7$, then $\Gamma$ and $\Gamma'$ are conjugate in $\Ot(d+1)$, and consequently the spherical orbifolds $S^d/\Gamma$ and $S^d/\Gamma'$ are isometric. 
\end{theorem}

\begin{proof}
Let $\Gamma$ be a finite subgroup of $\Ot(d+1)$ with $|\Gamma|\leq 7$. 
The goal is to show that $F_{\Gamma}(z)$ determines $\Gamma$ up to conjugation in $\Ot(d+1)$. 
We divide the proof in terms of the number of elements in $\Gamma$, which is a spectral invariant (see Remark~\ref{rem:volume}). 
The cases $|\Gamma|=1,2,3$ were established in \cite[Prop.~4.1]{LauretLinowitz}.

\medskip\noindent$\bullet$
$|\Gamma|=4$. 
As an abstract group there are two possibilities:  $\Gamma\simeq\Z/4\Z$ (i.e., $\Gamma$ is cyclic) or $\Gamma\simeq(\Z/2\Z)^2$. 
We first determine $F_{\Gamma}(z)$ in terms of the eigenvalues of the elements in $\Gamma$ in each case. 

Suppose first that $\Gamma$ is cyclic. 
Let $\gamma_0$ be any element in $\Gamma$ of order $4$. 
The eigenvalues of the matrix $\gamma_0$ are necessarily as follows:
$\mi$ and $-\mi$ have multiplicity $m_1$,
$-1$ has multiplicity $m_2$, 
and $1$ has multiplicity $d+1-2m_1+m_2$ for some $m_1,m_2\in\N_0=\N\cup\{0\}$ satisfying $d+1-2m_1+m_2\geq0$. 
It follows from \eqref{eq:Molien} that 
\begin{equation}\label{eq4:ciclico}
\begin{aligned}
\frac{4}{1-z^2}F_{\Gamma}(z)&
= \frac{1}{(1-z)^{d+1}}
+\frac{1}{(1+z)^{2m_1}(1-z)^{d+1-2m_1}}
\\ & \quad
+\frac{2}{(1-z\mi)^{m_1} (1+z\mi)^{m_1} (1+z)^{m_2} (1-z)^{d+1-2m_1-m_2}}
.
\end{aligned}
\end{equation}

We now suppose that $\ref{eq4:nociclico}Gamma\simeq(\Z/2\Z)^2$. 
Let $\gamma_1$ and $\gamma_2$ be two distinct non-trivial elements in $\Gamma$; thus $\Gamma=\{\Id_{d+1},\gamma_1,\gamma_2,\gamma_3:=\gamma_1\gamma_2\}$. 
For any $j\in\{1,2,3\}$, $\gamma_j$ has order $2$; we denote by $m_j'$ the multiplicity of the eigenvalue $-1$ of $\gamma_j$. 
We note that for $i,j,k\in\Z$ satisfying $\{i,j,k\}=\{1,2,3\}$, since $\gamma_i\gamma_j=\gamma_k$, one can check that $|m_i'-m_j'|\leq m_k'\leq m_i'+m_j'$ and $m_1'+m_2'+m_3'\equiv0\pmod 2$.
We have that
\begin{equation}\label{eq4:nociclico}
\begin{aligned}
\frac{4}{1-z^2}F_{\Gamma}(z)&
= \sum_{k=0}^3 \frac{1}{\det(1-z\gamma_k)}
= \frac{1}{(1-z)^{d+1}}
+\sum_{j=1}^3 \frac{1}{(1+z)^{m_j'}(1-z)^{d+1-m_j'}}
.
\end{aligned}
\end{equation}

We observe from \eqref{eq4:ciclico} and \eqref{eq4:nociclico} that, for an arbitrary subgroup $\Gamma$ of $\Ot(d+1)$ having four elements, $F_{\Gamma}(z)$ has a pole at $z=\mi$ if and only if $\Gamma$ is cyclic. 
Consequently, $F_{\Gamma}(z)$ determines whether $\Gamma\simeq\Z/4\Z$ or $\Gamma\simeq (\Z/2\Z)^2$. 

Assume that $\Gamma$ is cyclic. 
It is evident from \eqref{eq4:ciclico} that the order of the pole of $F_{\Gamma}(z)$ at $z=\mi$ is $m_1$.
Hence $m_1$ is determined by $F_{\Gamma}(z)$, and so is the function
\begin{multline*}
\frac{4}{1-z^2}F_{\Gamma}(z)
- \frac{1}{(1-z)^{d+1}}
-\frac{1}{(1+z)^{2m_1}(1-z)^{d+1-2m_1}}
\\ 
=\frac{2}{(z-\mi)^{m_1} (z+\mi)^{m_1} (1+z)^{m_2} (1-z)^{d+1-2m_1-m_2}}
,
\end{multline*}
which has a pole of order $m_2$ at $z=-1$. 
We conclude that $F_{\Gamma}(z)$ determines the spectrum of the matrix $\gamma_0$, as well as $\gamma_0$ and $\Gamma$ up to conjugation in $\Ot(d+1)$. 

We now assume that $\Gamma$ is not cyclic. 
By reordering the elements of $\Gamma=\{\Id_{d+1},\gamma_1,\gamma_2,\gamma_3\}$ if necessary, we can assume that $m_1'\geq m_2'\geq m_3'$. 
It is clear from \eqref{eq4:nociclico} that the function $\frac{4}{1-z^2}F_{\Gamma}(z) - \frac{1}{(1-z)^{d+1}}$, which is determined by $F_{\Gamma}(z)$, has a pole at $z=-1$ of order $m_1'$. 
The $(-m_1')$-th term in its Laurent series is given by 
\begin{align*}
\lim_{z\to-1} (1+z)^{m_1'}\left(\frac{4}{1-z^2}F_{\Gamma}(z) - \frac{1}{(1-z)^{d+1}}\right) 
=
\begin{cases}
\frac1{2^{d+1-m_1'}}
	&\text{if }m_1'>m_2',\\
\frac1{2^{d-m_1'}}
	&\text{if }m_1'=m_2'>m_3',\\
\frac{3}{2^{d+1-m_1'}}
	&\text{if }m_1'=m_2'=m_3'.
\end{cases}
\end{align*}
Since $m_1'$ is determined by $F_{\Gamma}(z)$ and the three values in the three rows above are different, it follows that $F_{\Gamma}(z)$ also determines which condition holds.

In the first case, $m_1'>m_2'$, $F_{\Gamma}(z)$ determines
the function
$$
\frac{4}{1-z^2}F_{\Gamma}(z)
-\frac{1}{(1-z)^{d+1}}
-
\frac{1}{(1+z)^{m_1'}(1-z)^{d+1-m_1'}}
=\sum_{j=2}^3 \frac{1}{(1+z)^{m_j'}(1-z)^{d+1-m_j'}}
,
$$
so the same argument as above shows that $F_{\Gamma}(z)$ also determines $m_2'$ and $m_3'$.
In the second case, $m_1'=m_2'>m_3'$, $m_1'$ and $m_2'$ are determined, which gives immediately that so is $m_3'$. 
The third case is obvious. 
It is clear that the values of $m_1',m_2',m_3'$ determine $\Gamma$ up to conjugation in $\Ot(d+1)$, as required.

\medskip\noindent$\bullet$
$|\Gamma|=5$. 
It turns out that $\Gamma$ is cyclic, that is, $\Gamma\simeq \Z/5\Z$. 
Let $\gamma_0$ be any non-trivial element in $\Gamma$. 
In this case, it is adequate to use the language of Subsection~\ref{subsec:lens}. 
Since $\gamma_0$ has order odd, then $\gamma_0\in\SO(d+1)$, and consequently $S^{d}/\Gamma$ is isometric to a lens orbifold.
Therefore, there are non-negative integers $n_0,n_1,n_2$ satisfying $n=\lfloor\frac{d+1}{2}\rfloor=n_0+n_1+n_2$ such that $S^d/\Gamma$ is isometric to 
$$
L:=L_d(5;0^{n_0},1^{n_1},2^{n_2}).
$$
From \eqref{eq:MolienLens}, it follows that
\begin{align*}
\frac{5}{1-z^2}F_{\Gamma}(z)  &
=\frac{1}{(1-z)^{d+1}}+
\frac{2}{(1-z)^{d+1-2(n_1+n_2)} (1-z\xi_5)^{n_1}(1-z\xi_5^4)^{n_1}	(1-z\xi_5^2)^{n_2}(1-z\xi_5^3)^{n_2} }
\\ &
+ \frac{2}{(1-z)^{d+1-2(n_1+n_2)} (1-z\xi_5)^{n_2} (1-z\xi_5^4)^{n_2} (1-z\xi_5^2)^{n_1}(1-z\xi_5^3)^{n_1}}
. 
\end{align*}

The function $\widetilde F_{\Gamma}(z):=
\frac{5}{1-z^2}F_{\Gamma}(z) -\frac{1}{(1-z)^{d+1}} $ is clearly determined by $F_{\Gamma}(z)$ (recall that $d$ is determined by $F_\Gamma$ by Remark~\ref{rem:volume}). 
It has a pole of order $d+1-2(n_1+n_2)$ at $z=1$ since 
\begin{multline} \label{eq5:degreepole}
\lim_{z\to1}(1-z)^{d+1-2(n_1+n_2)}\,\widetilde F_{\Gamma}(z)
= \lim_{z\to1} \left(
\frac{2}{ (z-\xi_5)^{n_1}(z-\xi_5^4)^{n_1}	(z-\xi_5^2)^{n_2}(z-\xi_5^3)^{n_2} }
\right.
\\ 
+\left. \frac{2}{(z-\xi_5)^{n_2} (z-\xi_5^4)^{n_2} (z-\xi_5^2)^{n_1}(z-\xi_5^3)^{n_1}}\right)
\\  
=
\frac{2}{\big(2-2\cos(\tfrac{2\pi}{5})\big)^{n_1} \big(2-2\cos(\tfrac{4\pi}{5})\big)^{n_2}}
+\frac{2}{\big(2-2\cos(\tfrac{2\pi}{5})\big)^{n_2} \big(2-2\cos(\tfrac{4\pi}{5})\big)^{n_1}}>0
.
\end{multline}
We deduce that $n_1+n_2$ is determined by $F_{\Gamma}(z)$.

Note that $F_{\Gamma}(z)$ determines 
\begin{multline}\label{eq:verification}
(1-z)^{d+1-2(n_1+n_2)} F_{\Gamma}(z)
=\frac{1-z^2}{5} 
\left(
	\frac{2}{ (1-z\xi_5)^{n_1}(1-z\xi_5^4)^{n_1}	(1-z\xi_5^2)^{n_2}(1-z\xi_5^3)^{n_2} }
\right.
\\ 
\left.
	+ \frac{2}{(1-z\xi_5)^{n_2} (1-z\xi_5^4)^{n_2} (1-z\xi_5^2)^{n_1}(1-z\xi_5^3)^{n_1}}
	+ \frac{1}{(1-z)^{2(n_1+n_2)}}
\right)
\\
=\frac{1-z^2}{5} \sum_{k=0}^{4} 
	\frac{1}{(1-z\xi_q^{k})^{n_1} (1-z\xi_q^{4k})^{n_1} (1-z\xi_q^{2k})^{n_2} (1-z\xi_q^{3k})^{n_2}}
	,
\end{multline}
which is precisely the generating function associated to the $(2n_1+2n_2-1)$-dimensional lens space $L(5;1^{n_1},2^{n_2})$ by \eqref{eq:MolienLens}. 
Proposition~\ref{prop:varphi(q)<=6} ensures that $n_1,n_2$ are determined by $F_{\Gamma}(z)$. 
We conclude that $F_{\Gamma}(z)$ determines the whole spectrum of $\gamma_0$.
Since any element in $\Ot(d+1)$ is conjugate to a diagonal matrix over the complex number whose entries are its eigenvalues, $F_{\Gamma}(z)$ determines $\gamma_0$ and $\Gamma$ up to conjugation in $\Ot(d+1)$, which is the requested assertion.

\medskip\noindent$\bullet$
$|\Gamma|=6$. 
We have two possibilities:  $\Gamma\simeq\Z/6\Z\simeq \Z/2\Z\oplus\Z/3\Z$ or $\Gamma\simeq\mathbb S_3$ (the symmetric group in three letters). 
We next determine $F_{\Gamma}(z)$ in each case independently. 

We first suppose that $\Gamma$ is cyclic. 
Let $\gamma_0$ be any generator, whose spectrum is as follows: 
$\xi_6,\xi_6^5$ have multiplicity $m_1$, 
$\xi_6^2,\xi_6^4$ have multiplicity $m_2$, 
$-1=\xi_6^3$ has multiplicity $m_3$, 
$+1$ has multiplicity $d+1-2m_1-2m_2-m_3$
Note that $m_1>0$, otherwise $\gamma_0$ would not have order $6$. 
It is a simple matter to deduce from \eqref{eq:Molien} that 
\begin{multline}\label{eq6:ciclico}
\frac{6}{1-z^2}F_{\Gamma}(z)
= \frac{1}{(1-z)^{d+1}}
+\frac{1}{(1-z)^{d+1-2m_1-m_3} (1+z)^{2m_1+m_3}  }
\\  \quad
+\frac{2}{(1-z)^{d+1-2m_1-2m_2} (1-z\xi_6^2)^{m_1+m_2} (1-z\xi_6^4)^{m_1+m_2} }
\\  \quad
+\frac{2}{(1-z)^{d+1-2m_1-2m_2-m_3} (1-z\xi_6)^{m_1} (1-z\xi_6^5)^{m_1} (1-z\xi_6^2)^{m_2} (1-z\xi_6^4)^{m_2} (1+z)^{m_3} }
.
\end{multline}

We now assume that $\Gamma\simeq\mathbb S_3$. 
In this case it is adequate to use the language of representation theory of finite groups. 
We denote by $\pi:\mathbb S_3\to\GL(d+1)$ the faithful representation of degree $d+1$ such that $\pi(\mathbb S_3)=\Gamma$. 
This representation can be decomposed as a sum of irreducible representations of $\mathbb S_3$, which are as follows:
\begin{itemize}
\item $\id:\mathbb S_3\to\GL(\R)$, $\id(\gamma)=1$ for all $\gamma$. 

\item $\rho_1:\mathbb S_3\to\GL(\R)$, $\rho_1(\gamma)=\op{sgn}(\gamma)$, the sign of the permutation $\gamma$. 

\item $\rho_2:\mathbb S_3\to\GL(V)$, where $V=\{(x_1,x_2,x_3)\in\R^3: x_1+x_2+x_3=0\}$,  $\rho_2(\gamma)\cdot (x_1,x_2,x_3)= (x_{\gamma(1)},x_{\gamma(2)},x_{\gamma(3)})$. 
\end{itemize}
Since $\id$ and $\rho_1$ have dimension one and $\rho_2$ has dimension two, there are non-negative integers $m_1',m_2'$ such that 
\begin{equation}\label{eq:pi}
\pi=(d+1-m_1'-2m_2')\,\id + m_1'\,\rho_1 + m_2'\,\rho_2. 
\end{equation}
Note that $m_2'>0$ since otherwise $\pi$ is not faithful.

We write 
$
\mathbb S_3
=\langle \sigma,\tau\mid  \sigma^3=e,\, \tau^2=e,\, \tau\sigma\tau=\sigma^{-1} \rangle 
=\{e,\sigma,\sigma^2,\tau,\tau\sigma,\tau\sigma^2\}
.
$
One clearly has $\rho_1(\sigma^j)=1$ and $\rho_1(\tau \sigma^j)=-1$ for $j=0,1,2$.
Furthermore, $\rho_2(\sigma)$ and $\rho_2(\sigma^2)$ are rotations in $V$ by $2\pi/3$ and $-2\pi/3$, respectively, and  $\rho_2(\tau),\rho_2(\tau\sigma),\rho_2(\tau\sigma^2)$ are reflections. 
Therefore, 
$\rho_2(\sigma)$ and $\rho_2(\sigma^2)$ have eigenvalues $\xi_3$ and $\xi_3^2$,
$\rho_2(\tau)$, $\rho_2(\tau\sigma)$, and $\rho_2(\tau\sigma^2)$ have eigenvalues $+1,-1$, 
and $\rho_2(e)$ has the eigenvalue $1$ with multiplicity $2$. 
Combining this information and \eqref{eq:pi}, we obtain the following:
\begin{itemize}
\item $\pi(e)$ has the eigenvalue $1$ with multiplicity $d+1$,

\item $\pi(\sigma)$ and $\pi(\sigma^2)$ have eigenvalues $\xi_3,\xi_3^2$ with multiplicity $m_2'$, and $1$ with multiplicity $d+1-2m_2'$,

\item $\pi(\tau\sigma^j)$ has eigenvalues $-1$ with multiplicity $m_1'+m_2'$ and $1$ with multiplicity $d+1-m_1'-m_2'$, for any $j=0,1,2$. 
\end{itemize}
Hence
\begin{equation}\label{eq6:simetrico}
\begin{aligned}
\frac{6}{1-z^2}F_{\Gamma}(z)&
= \frac{1}{(1-z)^{d+1}}
+\frac{3}{(1-z)^{d+1-m_1'-m_2'} (1+z)^{m_1'+m_2'}}
\\ & \quad
+\frac{2}{(1-z)^{d+1-2m_2'} (1-z\xi_3)^{m_2'} (1-z\xi_3^2)^{m_2'} }
.
\end{aligned}
\end{equation}

We observe from \eqref{eq6:ciclico} and \eqref{eq6:simetrico} that, for an arbitrary subgroup $\Gamma$ of $\Ot(d+1)$ having six elements, $F_{\Gamma}(z)$ has a pole at $z=\xi_6$ if and only if $\Gamma$ is cyclic. 
Consequently, $F_{\Gamma}(z)$ determines whether $\Gamma\simeq\Z/6\Z$ or $\Gamma\simeq \mathbb S_3$. 

We first suppose that $\Gamma$ is cyclic. 
It follows immediately from \eqref{eq6:ciclico} that $\frac{6}{1-z^2} F_{\Gamma}(z)$ has the following: 
a pole at $z=\xi_6$ of order $m_1$, 
a pole at $z=\xi_6^2$ of order $m_1+m_2$, 
a pole at $z=-1$ of order $2m_1+m_3$.
This forces $m_1,m_2,m_3$ to be determined by $F_{\Gamma}(z)$, consequently $\Gamma$ is also determined up to conjugation in $\Ot(d+1)$. 

We now consider the case $\Gamma\simeq\mathbb S_3$, that is, $z=\xi_6$ is not a pole of $F_{\Gamma}(z)$. 
It is clear from \eqref{eq6:simetrico} that $F_{\Gamma}(z)$ has a pole at $z=\xi_3$ of order $m_2'$. 
Thus, $F_{\Gamma}(z)$ determines the function 
\begin{multline*}
\frac{6}{1-z^2}F_{\Gamma}(z)
- \frac{1}{(1-z)^{d+1}}
-
\frac{2}{(z-\xi_3)^{m_2'} (z-\xi_3^2)^{m_2'} (1-z)^{d+1-2m_2'}}
\\ 
=\frac{3}{(1+z)^{m_1'+m_2'}(1-z)^{d+1-m_1'-m_2'}}
,
\end{multline*}
which has a pole at $z=-1$ of order $m_1'+m_2'$.
Hence $m_1'$ and $m_2'$ are both determined by $F_{\Gamma}(z)$. 
By \eqref{eq:pi}, the representation $\pi$ is determined and consequently also is $\Gamma$ up to conjugation in $\Ot(d+1)$.

\medskip\noindent$\bullet$
$|\Gamma|=7$. 
Let $\gamma_0$ be any non-trivial element in $\Gamma$. 
In this case, it is adequate to use the language of Subsection~\ref{subsec:lens}. 
Since $\gamma_0$ has order odd, then $\gamma_0\in\SO(d+1)$, and consequently $S^{d}/\Gamma$ is isometric to a lens orbifold.
Therefore, there are non-negative integers $n_0,n_1,n_2,n_3$ satisfying $n=n_0+\dots+n_3$ such that $S^d/\Gamma$ is isometric to 
$$
L:=L_d(7;0^{n_0},1^{n_1},2^{n_2},3^{n_3}).
$$

It follows from \eqref{eq:MolienLens} that 
\begin{multline*}
\widetilde F_{\Gamma}(z):=
\frac{7}{1-z^2}F_{\Gamma}(z) 
-\frac{1}{(1-z)^{d+1}} 
\\
= 
 \frac{2}{(1-z)^{d+1-2(n_1+n_2+n_3)} 
	(1-z\xi_7)^{n_1}   (1-z\xi_7^2)^{n_2}
	(1-z\xi_7^3)^{n_3} (1-z\xi_7^4)^{n_3}
	(1-z\xi_7^5)^{n_2} (1-z\xi_7^6)^{n_1} }
\\
+ \frac{2}{(1-z)^{d+1-2(n_1+n_2+n_3)} 
	(1-z\xi_7)^{n_3}   (1-z\xi_7^2)^{n_1}
	(1-z\xi_7^3)^{n_2} (1-z\xi_7^4)^{n_2}
	(1-z\xi_7^5)^{n_1} (1-z\xi_7^6)^{n_3} }
\\
+ \frac{2}{(1-z)^{d+1-2(n_1+n_2+n_3)} 
	(1-z\xi_7)^{n_2}   (1-z\xi_7^2)^{n_3}
	(1-z\xi_7^3)^{n_1} (1-z\xi_7^4)^{n_1}
	(1-z\xi_7^5)^{n_3} (1-z\xi_7^6)^{n_2} }
. 
\end{multline*}
By using the same argument as in \eqref{eq5:degreepole} for the case $|\Gamma|=5$, one can show that this function, which is determined by $F_{\Gamma}(z)$, has a pole of order $d+1-2(n_1+n_2+n_3)$ at $z=1$. 
Hence, $F_{\Gamma}(z)$ determines $n_1+n_2+n_3$; therefore so is $n_0$.

With a computation analogous to \eqref{eq:verification}, one can see that  $(1-z)^{d+1-2(n_1+n_2+n_3)}F_{\Gamma}(z)$, which is determined by $F_{\Gamma}(z)$, coincides with the generating function associated to the $(2n_1+2n_2+2n_3-1)$-dimensional lens space $L(7;1^{n_1},2^{n_2},3^{n_3})$. 
Proposition~\ref{prop:varphi(q)<=6} yields that $n_1,n_2,n_3$ are determined by $F_{\Gamma}(z)$.
We conclude that $F_{\Gamma}(z)$ determines the spectrum of $\gamma_0$, hence $\Gamma$ is determined up to isometry in $\Ot(d+1)$ by $F_{\Gamma}(z)$, as asserted. 
\end{proof}

\section{Highest volume of spherical space forms with non-cyclic fundamental groups}\label{sec:errata}

Lemma~4.4 in \cite{LauretLinowitz} classifies the only possible non-cyclic fundamental groups of spherical space forms with order strictly less than 24, namely,
\begin{itemize}
\item the quaternion group $Q_8:=\langle B,R\mid B^4=e,\, R^2=B^2,\, RBR^{-1}=B^3\rangle$ of order $8$,

\item the group $P_{12}:=\langle A,B\mid A^3=B^4=e,\, BAB^{-1}=A^2\rangle$ of order $12$,

\item the generalized quaternion group $Q_{16}:=\langle B,R\mid B^8=e,\, R^2=B^4,\, RBR^{-1}=B^7\rangle$ of order $16$,

\item the group $P_{20}:=\langle A,B\mid A^5=B^4=e,\, BAB^{-1}=A^4\rangle$ of order $20$. 
\end{itemize}
Furthermore, it claims that 
\begin{quote}
\it for each group $H$ in the above list and $m\in\N$, there is up to isometry exactly one spherical space form with fundamental group isomorphic to $H$ and dimension $4m-1$.
\end{quote}
This last part is incorrect and the goal of this section is to correct it. 
The next statement replaces it. 

\begin{lemma}\label{lem:errata}
Let $H$ be a group in the above list and $m\in\N$.  Two isospectral $(4m-1)$-dimensional spherical space forms with fundamental groups isomorphic to $H$ are isometric. 
\end{lemma}

Only the last paragraph in the proof of \cite[Lem.~4.4]{LauretLinowitz} is incorrect. 
It obtains the false fact in the previous statement from \cite[Prop.~5.3]{Wolf01}. 
We next fix it modifying only the erroneous parts.  

\begin{proof}
The ways that a fixed point free group $H$ embeds into $\SO(2n)$ via complex fixed point free representations to give all the isometry classes of $(2n-1)$-dimensional spherical space forms is quite complicated to be explained. 
See \cite[\S4--5]{Wolf01} for a nice summary, or \cite[\S7.2--3]{Wolf-book} for their classifications. 
The main point is the following: 
\begin{quote}
given two $d$-dimensional spherical space forms $S^{d}/\rho_1(\Gamma)$ and $S^{d}/\rho_2(\Gamma)$ ($\Gamma$ is a fixed point group and $\rho_1,\rho_2:H\to\SO(d+1)$ are fixed point free representations), they are isometric if and only if there is an automorphism $\Phi$ of $\Gamma$ such that $\rho_1\cong\rho_2\circ\Phi$ (i.e., $\rho_1$ and $\rho_2\circ\Phi$ are equivalent as representations of $\Gamma$). 
\end{quote}

We begin with the second group in the list, $H= P_{12}$, which is of Type I with parameters $m=3$, $n=4$, $r=2$, and $d=2$ in the notation in \cite{Wolf01}. 
By \cite[Prop.~4.19]{Wolf01}, there is only one irreducible complex fixed point free representation of $H$ up to equivalence, say $\pi$, which has dimension $2$.
This implies that, up to isometry, there is only one $(4m-1)$-dimensional spherical space form with fundamental group isomorphic to $H$, namely
\begin{equation}
S^{4m-1}/\big(\underbrace{(\pi\oplus\bar\pi)\oplus \dots \oplus (\pi\oplus\bar\pi) }_{m\text{-times}}\big)(H). 
\end{equation}
Of course, they are mutually non-isospectral because they have different dimensions. 

The case $H=Q_8$ follows in a very similar way since it has up to equivalence only one irreducible complex fixed point free representation of dimension $2$ by \cite[Prop.~4.19]{Wolf01}. 
Note that $Q_{8}$ is of Type II with parameters $m=1$, $n=4$, $r=1$, $d=1$, $s=1$, and $t=3$, according to the notation in \cite{Wolf01}. 

We now assume $H=P_{20}$. 
We have that $H$ is Type I, with parameters $m=5$, $n=4$, $r=4$, and $d=2$. 
In this case, \cite[Prop.~4.19]{Wolf01} yields there are two irreducible complex fixed point free representations of $H$ up to equivalence, which have dimension $2$.
A representative set is given by $\{\pi_{1,1}, \pi_{2,1}\}$, following the notation in \cite[(4.1)]{Wolf01}. 

The $3$-dimensional spherical space forms 
$S^{3}/(\pi_{1,1}\oplus\bar\pi_{1,1}) (H)$ and $S^{3}/(\pi_{2,1}\oplus\bar\pi_{2,1}) (H)$ are isometric by Proposition~\cite[Prop.~5.3]{Wolf01} since there is $\Phi\in\op{Aut}(H)$ such that $\pi_{1,1}\simeq\pi_{2,1}\circ\Phi$. 
By writing $\rho_h=h\pi_{1,1}\oplus(m-h)\pi_{2,1}$, we have that $\rho_h\simeq \rho_{m-h}\circ\Phi$ for integers $0\leq h\leq m$, 
hence 
$
S^{4m-1}/\big(\rho_h\oplus\bar\rho_h\big)(H)
$
and 
$
S^{4m-1}/\big(\rho_{m-h}\oplus\bar\rho_{m-h}\big)(H)
$
are isometric. 
Consequently, the isometry classes of all $(4m-1)$-dimensional spherical space forms with fundamental groups isomorphic to $H$ are 
\begin{equation}\label{eq:S^4m-1/rho_h(H)}
S^{4m-1}/(\rho_h\oplus\bar\rho_h)(H)
\qquad\text{for }h=0,\dots,\lfloor\tfrac{m}{2}\rfloor. 
\end{equation}
It remains to show that they are pairwise non-isospectral. 
To do that, we next prove that $F_{(\rho_h\oplus\bar\rho_h)(H)}^{(10)}(z)$ determines $h$. 

We have the presentation
\begin{equation}
H=\langle A,B\mid A^5=B^4=e,\; BAB^{-1}=A^4\rangle.
\end{equation}
The elements in $H$ with order $10$ are $(AB^2)^j$ with $\gcd(j,10)=1$. 
One obtains from \cite[(4.1)]{Wolf01} that 
\begin{align}
\pi_{k,1}(A)&=
\begin{pmatrix}
\xi_5^{k}&\\ & \xi_5^{-k}
\end{pmatrix}
,&
\pi_{k,1}(B^2)=
\begin{pmatrix}
-1&\\ & -1
\end{pmatrix}
;
\end{align}
thus the multiset of eigenvalues of the $4\times4$-matrix $(\pi_{k,1}\oplus \bar\pi_{k,1})\big((AB^2)^j\big)$ is  \begin{equation}
\{\!\{\xi_{10}^{(2k+5)j},\xi_{10}^{(2k+5)j}, \xi_{10}^{-(2k+5)j}, \xi_{10}^{-(2k+5)j}\}\!\}
. 
\end{equation}
Hence
\begin{equation}
\begin{aligned}
F_{(\rho_h\oplus\bar\rho_h)(H)}^{(10)}(z)&
= \sum_{j\in\{1,3,7,9\}} \frac{1}{ 
	(z-\xi_{10}^{3j})^{2h} 
	(z-\xi_{10}^{-3j})^{2h} 
	(z-\xi_{10}^{j})^{2(m-h)}
	(z-\xi_{10}^{-j})^{2(m-h)}
}
\\ & 
= \frac{2}{ 
	(z-\xi_{10}^{3})^{2h} 
	(z-\xi_{10}^{-3})^{2h} 
	(z-\xi_{10})^{2(m-h)}
	(z-\xi_{10}^{-1})^{2(m-h)}
}
\\&\quad 
+
\frac{2}{ 
	(z-\xi_{10})^{2h} 
	(z-\xi_{10}^{-1})^{2h} 
	(z-\xi_{10}^{3})^{2(m-h)}
	(z-\xi_{10}^{-3})^{2(m-h)}
}
,
\end{aligned}
\end{equation}
which has a pole of order $2(m-h)$ at $z=\xi_{10}$. 
We conclude that $F_{(\rho_h\oplus\bar\rho_h)(H)}(z)$ determines $h$, as asserted.

We now assume $H=Q_{16}$. 
The proof is very similar to the one for $P_{20}$, so we omit several details. 
The group $H$ is of Type II with parameters $m=1$, $n=8$, $r=1$, $d=1$, $s=1$, and $t=7$, according to \cite{Wolf01}. 
The complex fixed point free representations $\alpha_{1,1}$ and $\alpha_{1,3}$ (see \cite[(4.2)]{Wolf01}) of $H$ are irreducible, non-equivalent, and have dimension $2$.
The isometry classes of $(4m-1)$-dimensional spherical space forms with fundamental groups isomorphic to $H$ are as in \eqref{eq:S^4m-1/rho_h(H)} with $\rho_h=h\alpha_{1,1}\oplus(m-h)\alpha_{1,3}$. 
A simple computation gives 
\begin{equation}
\begin{aligned}
F_{(\rho_h\oplus\bar\rho_h)(H)}^{(8)}(z)&
= \frac{2}{ 
	(z-\xi_{8})^{2h} 
	(z-\xi_{8}^{-1})^{2h} 
	(z-\xi_{8}^3)^{2(m-h)}
	(z-\xi_{8}^{-3})^{2(m-h)}
}
\\&\quad 
+
\frac{2}{ 
	(z-\xi_{8}^{3})^{2h} 
	(z-\xi_{8}^{-3})^{2h} 
	(z-\xi_{8})^{2(m-h)}
	(z-\xi_{8}^{-1})^{2(m-h)}
}
,
\end{aligned}
\end{equation}
which has a pole of order $2(m-h)$ at $z=\xi_8$. 
This proves that $F_{(\rho_h\oplus\bar\rho_h)(H)}(z)$ determines $h$, and the proof is complete. 
\end{proof}

In \cite{LauretLinowitz}, Lemma~4.4, (essentially) used in the third line of page 701, combined with Prop.~4.3, imply that 
\begin{quote}
any pair of odd-dimensional isospectral and non-isometric spherical space forms $S^{2n-1}/\Gamma_1,S^{2n-1}/\Gamma_2$ with $|\Gamma_i|<24$ are necessarily lens spaces.
\end{quote}
This fact still follows from the new statement in Lemma~\ref{lem:errata}, so the correction does not affect any other result in \cite{LauretLinowitz}.

We end the article with another correction to \cite{LauretLinowitz}. 

\begin{remark}
The proof of Theorem~4.6 in \cite{LauretLinowitz} is not correct. 
The second identity in formula (38) is invalid since $\Phi_q(z)=\prod_{j=1}^{q_0} (z-\xi_q^{ht_j}) (z-\xi_q^{-ht_j})$ holds only if $\gcd(h,q)=1$. 

In order to solve it, one has to check that $F_L^{(k)}(z)=F_{L'}^{(k)}(z)$ for every divisor $k$ of $q$, instead of $F_L(z)=F_{L'}(z)$. 
Let $\gamma_0$ be a generator of the fundamental group $\Gamma$ of $L$. 
Clearly, for any divisor $k$ of $q$, the elements in $\Gamma$ with order $k$ are $\gamma_0^{h}$ with $0\leq h<q$ and $\gcd(h,q)=\frac{q}{k}$. 
Hence
\begin{equation}\label{eq:F_L(q;s+r.t)}
\begin{aligned}
F_{L}^{(k)}(z)&
= \frac{1-z^2}{q} \sum_{\substack{ 0\leq h\leq q-1,\\  \gcd(h,q)=\frac{q}{k}} } \frac{1}{ \prod_{i=1}^n (z-\xi_q^{hs_i})(z-\xi_q^{-hs_i})} \left(\frac{1}{\prod_{j=1}^{q_0} (z-\xi_q^{ht_j}) (z-\xi_q^{-ht_j})}\right)^r 
\\ &
= \frac{1-z^2}{q} \sum_{\substack{ 0\leq \ell\leq k-1,\\  \gcd(\ell,k)=1} } 
	\frac{1}{ \prod_{i=1}^n (z-\xi_q^{\frac{\ell q}{k} s_i})(z-\xi_q^{-\frac{\ell q}{k}s_i})} \left(\frac{1}{\prod_{j=1}^{q_0} (z-\xi_q^{\frac{\ell q}{k}t_j}) (z-\xi_q^{-\frac{\ell q}{k}t_j})}\right)^r 
\\ &
= \frac{1-z^2}{q} \sum_{\substack{ 0\leq \ell\leq k-1,\\  \gcd(\ell,k)=1} } 
	\frac{1}{ \prod_{i=1}^n (z-\xi_k^{\ell s_i}) (z-\xi_k^{-\ell s_i})} \left(\frac{1}{\prod_{j=1}^{q_0} (z-\xi_k^{\ell t_j}) (z-\xi_k^{-\ell t_j})}\right)^r 
\\ &
= \frac{1-z^2}{q} \frac{1}{\Phi_k(z)^{\frac{2q_0r}{\varphi(k)}}}
\sum_{\substack{ 0\leq \ell\leq k-1,\\  \gcd(\ell,k)=1} } 
	\frac{1}{ \prod_{i=1}^n (z-\xi_k^{\ell s_i}) (z-\xi_k^{-\ell s_i})}
\\ &
=  \frac{1}{\Phi_k(z)^{\frac{2q_0r}{\varphi(k)}}}
F_{L(q;s)}^{(k)}(z)
.
\end{aligned}
\end{equation}
Now, using that $F_{L(q;s)}^{(k)}(z)= F_{L(q;s')}^{(k)}(z)$ because $L(q;s)$ and $L(q;s')$ are isospectral by hypothesis, and returning over the steps in \eqref{eq:F_L(q;s+r.t)} for $s'$, we obtain $F_L^{(k)}(z)=F_{L'}^{(k)}(z)$ as required.
\hfill\qedsymbol 
\end{remark}

\bibliographystyle{plain}

\end{document}